\documentclass[preprint,12pt]{elsarticle}

\usepackage{amssymb}
\usepackage{url}
\usepackage{amsthm}
\usepackage{amsmath}

\newtheorem{dfn}{Definition}
\newtheorem{lem}{Lemma}
\newtheorem{obs}{Observation}
\newtheorem{thm}{Theorem}
\newtheorem{prp}{Property}
\newtheorem{pro}{Proposition}
\newtheorem{cor}{Corollary}

\journal{}

\begin{document}

\begin{frontmatter}

\title{The strong and doubly metric dimensions of Johnson and Kneser graphs}

\author{Jozef Kratica \fnref{mi}}
\ead{jkratica@mi.sanu.ac.rs}
\author{Mirjana \v{C}angalovi\'c \fnref{fon}}
\ead{mirjana.cangalovic@alumni.fon.bg.ac.rs}
\author{Vera Kova\v{c}evi\'c-Vuj\v{c}i\'c \fnref{fon}}
\ead{vera.vujcic@alumni.fon.bg.ac.rs}

 \address[mi]{Mathematical Institute, Serbian Academy of Sciences and Arts, Kneza Mihaila 36/III, 11 000 Belgrade, Serbia}
 \address[fon]{Faculty of Organizational Sciences, University of  Belgrade, Jove Ili\'ca 154, 11000 Belgrade, Serbia}

\begin{abstract}
In this paper, the strong and doubly metric dimensions of Johnson and Kneser graphs are considered. 
The exact value of the strong metric dimension of Johnson graph $J_{n,k}$ is obtained using
the well-known results from the literature. The strong metric dimension of Kneser graph
$K_{n,k}$ has been obtained for $n \ge 3k-1$. Finally, it has been shown that 
the doubly metric dimensions of Johnson graph $J_{n,2}$ and Kneser graph $K_{n,2}$
are both $\lceil \frac{2n}{3} \rceil$.  
\end{abstract}


\end{frontmatter}


\section{Introduction and previous work}

For any vertex $v \in V(G)$ of graph $G=(V(G),E(G))$ its open neighborhood
is the set $N(v) = \{w \in V(G) | vw \in E(G)\}$ and its closed neighborhood is
$N[v] = N(v) \bigcup \{v\}$. The degree of a vertex $v$, denoted by $deg(v)$, 
is defined as the cardinality of $N(v)$.
The maximum degree of $G$ is
$\Delta(G) = max \{deg(v) | v \in V(G)\}$ and its minimum degree is
$\delta(G) = min \{deg(v) | v \in V(G)\}$.
If $\Delta(G)=\delta(G)=r$, we say
that graph $G$ is $r$-regular.
The diameter of $G$ is $Diam(G) = max \{d(v,w) | v,w \in V(G)\}$. 

The concepts of resolving sets and the metric dimension of graph G were introduced 
independently by Slater (1975) in \cite{metd1} and by Harary and Melter (1976) in \cite{metd2}. 
For a simple connected undirected graph $G = (V(G),E(G))$ the distance between vertices $u$ and
$v$, i.e. the length of a shortest $u-v$ path, is denoted by $d(u,v)$. 
Vertex $w$ of graph $G$ is said to resolve vertices $u,v \in  V(G)$ if $d(u,w) \ne d(v,w)$. 
Set $R \subseteq V(G)$ is called a resolving set of $G$ if any pair of distinct vertices of $G$ 
is resolved by some vertex from $R$. The metric dimension of $G$, denoted by $\beta(G)$, 
is the minimum cardinality of its resolving sets, while the metric basis is any resolving set 
with cardinality $\beta(G)$. In \cite{metd3} it is proved that that the problem of computing 
the metric dimension of an arbitrary graph is NP-hard. 

The strong metric dimension of $G$, introduced in 2004 by Sebo and Tannier \cite{smetd1}, 
is a more restricted invariant than $\beta(G)$. A vertex $w$ strongly resolves two vertices 
$u$ and $v$ if $u$ belongs to a shortest $v-w$ path, or if $v$ belongs to a shortest $u-w$ path. 
A vertex set $S$ of $G$ is a strong resolving set of $G$ if every two distinct vertices of $G$ 
are strongly resolved by some vertex of $S$. The strong metric dimension of $G$, 
denoted by $\beta_S(G)$, is the minimum cardinality of its strong resolving sets. 
The strong metric basis of $G$ is any strong resolving set with the cardinality $\beta_S(G)$. 
Note that if a vertex $w$ strongly resolves vertices $u$ and $v$ then $w$ also resolves these vertices. 
Hence, every strong resolving set is a resolving set and $\beta(G) \le \beta_S(G)$. 
The problem of finding $\beta_S(G)$ is proved to be NP-hard \cite{smetd2}. 
A survey of the state-of-the-art (in 2013) on the strong metric dimension can be found in \cite{yuj}. 

In the sequel we will present some properties of the strong metric dimension from literature which use the concepts 
of vertex covering and independence number of graphs. A vertex cover of graph $G$ is a set $C$ 
of vertices from $V(G)$ such that for each edge $uv \in E(G)$ it holds $u \in C$ or $v \in C$.
The vertex covering number of $G$, denoted by $vc(G)$, is the minimum cardinality of its vertex covers. 
On the other hand, the maximum cardinality of a set of vertices of $G$, no two of which are adjacent, is called
the independence number of graph $G$ and is denoted by $ind(G)$. The clique number of $G$, denoted by $\omega(G)$,
is the cardinality of a maximum complete subgraph (clique) of graph $G$. It is well-known that $ind(G) = \omega(\overline{G})$,
where $\overline{G}$ represents the complement of graph $G$, i.e. $V(\overline{G}) = V(G)$ and 
$E(\overline{G}) = \overline{E(G)}$.

A well-known result obtained by Gallai \cite{vc1} states:

\begin{pro} \label{vc1a} \mbox{\rm(\cite{vc1})} For any graph $G$ without isolated vertices, it holds 
$vc(G) = |V(G)| - ind(G)$.
 \end{pro}

The concept of the strong resolving graph $G_{SR}$, introduced in \cite{smetd2}, is defined as follows: 

\begin{dfn} \label{gsr1a} \mbox{\rm(\cite{smetd2})} For any two vertices $u$ and $v$ from graph $G$, we say they are 
{\bf mutually maximally distant} (denoted by $u \, MMD \, v$) if 
$$(\forall w_1 \in N(u)) \, d(v,w_1) \le d(u,v) \wedge (\forall  w_2 \in N(v)) \, d(u,w_2) \le d(u,v)$$ holds. 
\end{dfn} 

\begin{dfn} \label{gsr1b} \mbox{\rm(\cite{smetd2})} Let $G_{SR}$ be graph obtained from $G$, such that \; 
$V(G_{SR}) = V(G)$ and $E(G_{SR}) = \{ uv \,|\, u \in V(G) \wedge v \in V(G) \wedge u \, MMD \, v\}$ 
\end{dfn}  

\begin{thm} \label{gsr1} \mbox{\rm(\cite{smetd2})} For any connected graph $G$ it holds $\beta_S(G) = vc(G_{SR})$.
 \end{thm}

Several articles on the strong metric dimension of graphs which are using $G_{SR}$ have been published.
Survey about the state of knowledge on the strong resolving graphs, as well as, some new results
regarding their properties, can be found in \cite{kuz18}. 

The concept of doubly resolving sets of $G$ has been introduced by Caceres et al. \cite{dmetd}, 
who proved that the metric dimension of the Cartesian product $G \Box G$ is tied in a strong sense 
to doubly resolving sets of $G$ with the minimum cardinality. Vertices $x$,$y$ of graph $G$ are 
said to doubly resolve vertices $u$,$v$ of G if $d(u,x)-d(u,y) \ne d(v,x)-d(v,y)$. A vertex set $D$ 
is a doubly resolving set of $G$ if every two distinct vertices of $G$ are doubly resolved by 
some two vertices in $D$. The doubly metric dimension of $G$, denoted by $\psi(G)$, is the minimum cardinality 
of its doubly resolving sets and the doubly metric basis of $G$ is any doubly resolving set of $G$ 
with cardinality $\psi(G)$. Computing $\psi(G)$ is also NP-hard \cite{ilpd}. A tight lower bound for the
doubly metric dimension of a graph has been proposed in \cite{kra26}.

Several other NP-hard graph invariants related to the metric dimension have been recently introduced, 
such as edge metric dimension $\beta_E(G)$ \cite{kel18}, mixed metric dimension $\beta_M(G)$ \cite{kel17}, 
equidistant dimension $eqdim(G)$ \cite{eqdim1}. We omit the definitions of these invariants 
because in this paper we restrict our analysis to the strong and the doubly metric dimension.

Since problems of computing the metric dimension and related invariants are NP-hard in
general case, many papers in the literature have been devoted to finding
the exact values or good bounds for some classes of graphs. In particular, the metric
dimension and resolving sets of Johnson and Kneser graphs have been studied in
\cite{metdj2,metdj1}. The equidistant dimension of Johnson and Kneser graphs is analyzed in 
\cite{eqdj}, while the edge  and mixed metric dimensions of Johnson graphs is the 
subject of \cite{emetdj}. This paper is devoted to the strong and the doubly metric dimensions of 
Johnson and Kneser graphs. We will also present integer linear programming formulations of the strong and the doubly 
metric dimension problems from literature, which will be used in order to obtain 
optimal values for small dimensions.

\subsection{Integer linear programming formulations from literature}

Given a simple connected undirected graph $G$ = ($V$,$E$), where
$V=\{1,2,\dots,n\}$, $|E|=m$, it is easy to determine the length
$d(u,v)$ of a shortest $u-v$ path for all $u,v \in V$, using any
shortest path algorithm.

First we present the integer linear programming (ILP) formulation of the strong metric dimension problem
\cite{ilps}. 

The coefficient matrix $A$ is defined as follows:

\begin{equation} \label{ilpsp}
A_{(u,v),i}  =\begin{cases}
1, & d(u,i) = d(u,v) + d(v,i)\\
1, & d(v,i) = d(v,u) + d(u,i)\\
0, & otherwise
\end{cases}
\end{equation}

where $1 \leq u < v \leq n, 1 \leq i \leq n$. Variable $y_i$
described by (2) determines whether vertex $i$ belongs to a strong
resolving set $S$.

\begin{equation}
y_i  =\begin{cases}
1, & i \in S\\
0, & i \notin S
\end{cases}
\end{equation}

The ILP model of the strong metric dimension problem can now be
formulated as:

\begin{equation} \label{ilpsk}
\min \sum\limits_{i = 1}^n {y_i }
\end{equation}

subject to:

\begin{equation}
\sum\limits_{i = 1}^{n} {A_{(u,v),i} \cdot y_i}  \ge 1\quad \quad
\quad \quad \quad \quad 1 \le u < v \le n
\end{equation}

\begin{equation} \label{ilpsk2}
y_i  \in \{ 0,1\} \quad \quad \quad \quad \quad 1 \le i \le n
\end{equation}

The next ILP formulation of the doubly metric dimension problem is propsed in 
\cite{ilpd}. 

The coefficient matrix $A$ is defined as
follows:

\begin{equation} \label{ilpdp}
A_{(u,v),(i,j)}  = \left\{ {\begin{array}{*{20}c}
   {1,\quad d(u,i) - d(v,i) \ne d(u,j) - d(v,j)}  \\
   {0,\quad d(u,i) - d(v,i) = d(u,j) - d(v,j)}  \\
\end{array}} \right.
\end{equation}

where $1 \leq u < v \leq n, 1 \leq i < j \leq n$. Variable $y_i$ described by (4)
determines whether vertex $i$ belongs to a doubly resolving set $D$. Similarly,
$x_{ij}$ determines whether both $i,j$ are in $D$.

\begin{equation}
y_i  = \left\{ {\begin{array}{*{20}c}
   {1,\quad i \in D}  \\
   {0,\quad i \notin D}  \\
\end{array}} \right.
\end{equation}

\begin{equation} 
x_{ij}  = \left\{ {\begin{array}{*{20}c}
   {1,\quad \quad i,j \in D}  \\
   {0,\quad otherwise}  \\
\end{array}} \right.
\end{equation}

The ILP model of the doubly metric dimension problem can now be formulated as:

\begin{equation} \label{ilpdk}
\min \quad \sum\limits_{k = 1}^n {y_k }
\end{equation}

subject to:

\begin{equation}
\sum\limits_{i = 1}^{n - 1} {\sum\limits_{j = i + 1}^n
{A_{(u,v),(i,j)}  \cdot x_{ij} } }  \ge 1\quad \quad \quad \quad
\quad \quad 1 \le u < v \le n
\end{equation}

\begin{equation}
x_{ij}  \le \frac{1}{2}y_i  + \frac{1}{2}y_j \quad \quad \quad \quad \quad \quad
\quad \quad \quad \quad \quad 1 \le i < j \le n
\end{equation}

\begin{equation}
x_{ij}  \ge y_i  + y_j  - 1\quad \quad \quad \quad \quad \quad \quad \quad \quad
\quad \quad 1 \le i < j \le n
\end{equation}

\begin{equation} \label{ilpdk2}
x_{ij}  \in \{ 0,1\} ,\;y_k  \in \{ 0,1\} \quad \quad \quad \quad
\quad 1 \le i < j \le n,\;1 \le k \le n
\end{equation}

\subsection{Johnson and Kneser graphs and related results from literature}

Let $n$ and $k$ be positive integers ($n > k$) and $[n] = \{1,2,...,n\}$.
Then $k$-subsets are subsets of $[n]$ which have cardinality equal to $k$.
Johnson graph $J_{n,k}$ is an undirected graph defined on all $k$-subsets of set $[n]$ as vertices,
where two $k$-subsets are adjacent if their intersection has cardinality equal to $k-1$.
Formally, $V(J_{n,k}) = \{ A | A \subset [n], |A|=k\}$ and
$E(J_{n,k}) = \{ AB | A,B \subset [n], |A|=|B|=k, |A \bigcap B|=k-1, lex(A) < lex(B)\}$,
where $lex$ is any lexicographic order of $k$-subsets of $[n]$.

It is easy to see that $J_{n,k}$ and $J_{n,n-k}$ are isomorphic, so
we shall only consider Johnson graphs with $n \ge 2k$.
The distance between two vertices $A$ and $B$ in $J_{n,k}$
can be computed by Property \ref{distj}.

\begin{prp} \label{distj} For $A,B \in V(J_{n,k})$ it holds 
$d(A,B) = |A \setminus B| = |B \setminus A| = k-|A \bigcap B|$.
\end{prp}

In a special case when $n=2k$ distance between $\overline{A} = [n] \setminus A$ and $B$ can be
computed by Property \ref{comp}.

\begin{prp} \label{comp} For $A,B \in V(J_{2k,k})$ it holds $d(\overline{A},B) = k - d(A,B) = |A \bigcap B|$
\end{prp}
 
Considering  Property \ref{distj}, it is easy to see that Johnson graph $J_{n,k}$ is a $k (n-k)$-regular graph 
of diameter $k$. 

Kneser graph $K_{n,k}$ is an undirected graph also defined on all $k$-subsets of set $[n]$ as vertices,
where two $k$-subsets are adjacent if their intersection is the empty set.
Formally, $V(K_{n,k}) = \{ A | A \subset [n], |A|=k\}$ and
$E(K_{n,k}) = \{ AB | A,B \subset [n], |A|=|B|=k, A \bigcap B= \emptyset, lex(A) < lex(B)\}$.

Kneser graph is connected only if $n > 2 \cdot k$, it is also $\binom{n-k}{k}$-regular graph.
Specially, for $k=2$, Kneser graph $K_{n,2}$ is the complement of the corresponding 
Johnson graph $J_{n,2}$, and both graphs have diameter 2. Hence, if $d_{K_{n,2}}(A,B)=1$
then $d_{J_{n,2}}(A,B)=2$, and vice versa. Therefore, for $A \ne B$ it holds 
$d_{K_{n,2}}(A,B) = 3 - d_{J_{n,2}}(A,B)$.
 
Distances and diameter of Kneser graphs are investigated by Valencia-Pabon and Vera in \cite{pab05}:

\begin{equation} \label{diam1} 
Diam(K_{n,k}) = \lceil \frac{k-1}{n-2k} \rceil + 1.
\end{equation}

Consequently, if $n \ge 3k-1$ then 

\begin{equation} \label{diam2} 
Diam(K_{n,k}) = 2.
\end{equation}

In this case $d(A,B) = 2 \Leftrightarrow AB \notin E(K_{n,k}) \Leftrightarrow AB \in E((K_{n,k})_{SR})$. 
This implies $(K_{n,k})_{SR} = \overline{K_{n,k}}$.\\

Moreover, in case when $2k+1 \le n < 3k-1$ from (\ref{diam1}) it follows $Diam(K_{n,k}) \ge 3$.
In this case distances can be computed using the result by Stahl in \cite{sta76}:

\begin{lem} \label{sta} \mbox{\rm(\cite{sta76})} Let $A,B \in V(K_{n,k})$.
\begin{itemize}
\item[(i)] If $d(A,B)=2p$ then $|A \bigcap B| \ge k(2p+1)-pn$;
\item[(ii)] If d(A,B)=2p+1 then $|A \bigcap B| \le p(n-2k)$.
\end{itemize}
\end{lem}

Moreover, Valencia-Pabon and Vera in \cite{pab05} unify the formula for computing these distances:

\begin{equation} \label{dkn} 
d(A,B) = min \{2 \lceil \frac{k-s}{n-2k} \rceil, 2 \lceil \frac{s}{n-2k} \rceil+1 \}
\end{equation}

where $s = |A \bigcap B|$.

The famous Erd\"os-Ko-Rado theorem, which was proved in 1938, but has not been published until 1961
(\cite{erd61}), says:

\begin{thm} \label{erd} \mbox{\rm(\cite{erd61})} For $n \ge 2k$ it holds $ind(K_{n,k}) = \binom{n-1}{k-1}$.
 \end{thm}

The following theorems give the exact values for the metric dimensions of $J_{n,2}$ and $K_{n,2}$
and an upper bound for the metric dimension of $J_{n,k}$ for $k \ge 3$.

\begin{thm} \label{mjk} \mbox{\rm(\cite{metdj1}, Theorem 3.32, Corrolary 3.33)} \\
For $n \ge 6$, it holds $\beta(J_{n,2}) = \beta(K_{n,2}) = \lceil \frac{2n}{3} \rceil$.
 \end{thm}

\begin{thm} \label{mjk1} \mbox{\rm(\cite{metdj2})} For $n \ge 2k$ and $k \ge 3$, it holds $\beta(J_{n,k}) \le \lfloor \frac{k \cdot (n+1)}{k+1} \rfloor$.
 \end{thm} 

Let us note that for $k=2$, using simple calculation by modulo 3, it holds $\lfloor \frac{2 \cdot (n+1)}{3} \rfloor = \lceil \frac{2n}{3} \rceil$.

The equidistant dimensions of Johnson and Kneser graphs have been studied
in \cite{eqdj} where the following theorems have been proved.

\begin{pro} \mbox{\rm(\cite{eqdj})} For $n \ge 6$ it holds $eqdim(J_{n,2}) = eqdim(K_{n,2}) = 3$ \end{pro}

\begin{pro} \mbox{\rm(\cite{eqdj})} For any odd $k \ge 3$ it holds $eqdim(J_{2k,k}) = \frac{1}{2} \cdot \binom{2k}{k}$. 
 \end{pro}

\begin{pro} \mbox{\rm(\cite{eqdj})} For $n \ge 9$ it holds  $eqdim(J_{n,3}) \le n-2$. 
 \end{pro} 

The exact values of the edge and mixed metric dimensions of Johnson graph $J_{n,2}$
have been obtained in \cite{emetdj} 

\begin{thm} \label{edj} \mbox{\rm(\cite{emetdj})} For $n \ge 5$ it holds $\beta_E(J_{n,2}) =  \beta_M(J_{n,2}) = \binom{n}{2} - \lfloor \frac{n}{2} \rfloor$.
 \end{thm}

\section{The strong metric dimension of Johnson and Kneser graphs}

\begin{lem} \label{sdimjl} Vertices $A,B \in V(J_{n,k})$ are mutually maximally distant if and only if $d(A,B) = k$.\end{lem}
\begin{proof}
($\Leftarrow$) If $d(A,B) = k = Diam(J_{n,k})$ it follows that all distances are at most $k$. Therefore,
$A$ and $B$ are mutually maximally distant. \\
($\Rightarrow$) Suppose the opposite, i.e. $d(A,B) < k = Diam(J_{n,k}) \Rightarrow A \bigcap B \ne \emptyset \Rightarrow
|A \bigcup B| < 2k \le n$. 
Then there exist $a \in A \bigcap B$ and $b \in V(J_{n,k}) \setminus (A \bigcup B)$. Let $C = (B \setminus \{a\}) \bigcup \{b\}$.
Then $A \bigcap C \subset A \bigcap B \Rightarrow |A \bigcap C| < |A \bigcap B| \Rightarrow d(A,C) = k - |A \bigcap C| > k - |A \bigcap B|
= d(A,B)$. Since $|C|=k$, $d(A,C) > d(A,B)$ and $|B \bigcap C| = |B|-1 = k-1 \Rightarrow d(B,C)=1 \Rightarrow C \in N(B)$,
so $A$ and $B$ are not mutually maximally distant in graph $J_{n,k}$.
Therefore, by contra-position, we have proved that if vertices $A,B \in V(J_{n,k})$ are mutually maximally distant then $d(A,B) = k$.
\end{proof}

Although the following results are known for more than a decade:
\begin{itemize}
\item Famous Erd\"os-Ko-Rado theorem (Theorem \ref{erd});
\item Gallai result from \cite{vc1} (Proposition \ref{vc1a}); 
\item Theorem from Oellermann et al. (\cite{smetd2} (Theorem \ref{gsr1}),
\end{itemize}
the next proposition uses them together, for the first time,
in order to obtain the strong metric dimension of Johnson graphs $J_{n,k}$.

\begin{pro} \label{sdimj} $\beta_S(J_{n,k}) = \binom{n-1}{k}$ \end{pro}
\begin{proof}
From Property \ref{distj} it follows that $Diam(J_{n,k}) = k$. 
Next, by Lemma \ref{sdimjl}, vertices $A,B \in V(J_{n,k})$ are mutually maximally distant if and only if
 $d(A,B) = k \Leftrightarrow A \bigcap B = \emptyset \Leftrightarrow AB \in E(K_{n,k})$.
Therefore, by Definition \ref{gsr1a} and Definition \ref{gsr1b}, we have $(J_{n,k})_{SR} = K_{n,k}$.

Finally, by Theorem \ref{erd}, Proposition \ref{vc1a} and Theorem \ref{gsr1} 
it follows $\beta_S(J_{n,k})$ = $vc((J_{n,k})_{SR})$ =  $vc(K_{n,k})$ =
$|K_{n,k}| - ind(K_{n,k})$ = $\binom{n}{k} - \binom{n-1}{k-1}$ =
$\binom{n-1}{k}$.
\end{proof}

It should be noted that Proposition \ref{sdimj} alternatively could be proved by 
using recently proposed outer general position number (\cite{tian25}).
Definition and properties of all four general position numbers are out of the scope of this paper
and can be found in \cite{gpp26}.

The following theorem gives the exact value of the strong metric dimension of
$K_{n,k}$ for $n \ge 3k-1$.

\begin{thm} \label{k3s} For $k \ge 2$ and $n \ge 3k-1$ it holds $\beta_S(K_{n,k}) = \binom{n}{k} - \lfloor \frac{n}{k} \rfloor$.
\end{thm}
\begin{proof}
Step 1. $\beta_S(K_{n,k}) \ge \binom{n}{k} - \lfloor \frac{n}{k} \rfloor$ \\
Suppose the contrary, i.e. $\beta_S(K_{n,k}) < \binom{n}{k} - \lfloor \frac{n}{k} \rfloor$.
Let $m = \lfloor \frac{n}{k} \rfloor$ and $S$ be a strong resolving base of $K_{n,k}$, so $|S| = \beta_S(K_{n,k}) < \binom{n}{k} - m$.
Let $A,B \in V(K_{n,k}) \setminus S$. Since $|V(K_{n,k}) \setminus S| \ge m+1$ then $(m+1) \cdot k > n$.
There exist at least $m+1$ vertices (sets) which are not in $S$, composed of at least $(m+1) \cdot k$ indices.
Since $\{1,2,...,n\}$ has $n$ indices, and $(m+1) \cdot k > n$ then there exists $A,B \notin S$ such
that $A \bigcap B \ne \emptyset$  which implies $AB \in V(\overline{K_{n,k}})$. Since $S$ is a
strong resolving set of $K_{n,k}$, and $A,B \notin S$ then 
$$(\exists C \in S) (d(C,A)+d(A,B)=d(C,B) \vee d(C,B)+d(B,A)=d(C,A))$$

Since $d(A,B)=2$ and $n \ge 3k-1$ according to (\ref{diam2}), $Diam(K_{n,k}) = 2 = Diam(\overline{K_{n,k}})$.
Then $C=A$ or $C=B$,
which is a contradiction with $A,B \notin S$!

Step 2. $\beta_S(K_{n,k}) \le \binom{n}{k} - \lfloor \frac{n}{k} \rfloor$ \\
Let $A_i = \{ik+1, ik+2, ...., (i+1)k\}$, for all $0 \le i \le m-1$,
let $S = V(K_{n,k}) \setminus \bigcup\limits_{i = 0}^{m - 1} {{A_i}}$.
For $A,B \notin S, A \ne B$ then $A=A_i$ and $B=A_j$, for some $i$ and $j$, such that $i \ne j$.
Since $A \bigcap B = A_i \bigcap A_j = \emptyset$ then $d(A,B) = 1$.
From $n \ge 2k+1$ it follows that there exists $l \notin A \bigcup B$.
Let $C = (B \setminus \{jk+1\}) \bigcup \{l\}$ then $|C|=k$, $B \bigcap C \ne \emptyset$
and $A \bigcap C = \emptyset$. Consequently, $d(B,C) = d(A,B) + d(A,C)$,
since $d(B,C)=2$ and $d(A,B) = d(A,C) = 1$.
It is obvious that $C \ne A_i$ for $0 \le i \le m-1$ and $|C|=k$ so
$C \in S$. Since for each pair $A,B \notin S, A \ne B$ there exists $C \in S$
such that $d(B,C) = d(A,B) + d(A,C)$, so set $S$ is a strongly resolving 
set of $K_{n,k}$. Since $|S|= \binom{n}{k} - m$ then $\beta_S(K_{n,k}) \le \binom{n}{k} - \lfloor \frac{n}{k} \rfloor$. 
\end{proof}

It should be noted that cases when $k \ge 3$ and $2k+1 \le n \le 3k-2$, are not covered by Theorem \ref{k3s},
since $Diam(K_{n,k}) \ge 3$. Table \ref{gurk} presents the values of $\beta_S(K_{n,k})$ for $3 \le k \le 5$
and $2k+1 \le n \le 3k-2$ obtained by Gurobi 13.0 using the ILP formulation (\ref{ilpsk})-(\ref{ilpsk2}).
Computations were performed on AMD Ryzen 5 5600H computer with 6 cores (12 logical processors) at 3.4 GHz and 8 GB of RAM. 
The same computer was also used for computations in Observation 1 and Observation 2.

\begin{table}
\caption{Values $\beta_S(K_{n,k})$, $3 \le k \le 5$ and $2k+1 \le n \le 3k-2$} \label{gurk}
 \small
\begin{center} 
\begin{tabular}{|c|c|c|c|c|c|}
\hline
$n$ & $k$ & $|V|$ & $|E|$ & $\beta_S(K_{n,k})$ & t (sec)\\
 \hline
7  & 3  & 35  & 70   & 30 & $ < 0.01$  \\
 \hline                                                  
9  & 4  & 126  & 315  & 115 & $ < 0.01$  \\
10 & 4  & 210  & 1575 & 182 & 1  \\
 \hline
11  & 5  & 462  & 1386  & 425 & 262  \\
12  & 5  & 792  & 8316  & 756 & 27692  \\
13  & 5  & 1287 & 36036 & $\le 1122$ & 28735  \\
 \hline
\end{tabular}
\end{center}
\end{table}

\section{The doubly metric dimension of $J_{n,2}$ and $K_{n,2}$}

The next theorem obtains the exact value of the doubly metric dimension of $J_{n,2}$.

\begin{thm} \label{j2} For $n \ge 15$ it holds  $\psi(J_{n,2}) = \lceil \frac{2n}{3} \rceil$.
\end{thm}
\begin{proof} By Theorem \ref{mjk} it holds $\beta(J_{n,2}) = \lceil \frac{2n}{3} \rceil$.
Having in mind that for every graph $G$, $\psi(G) \ge \beta(G)$ so $\psi(J_{n,2}) \ge \lceil \frac{2n}{3} \rceil$.
Let $S$ be a resolving set construncted in the proof of Theorem 3.32 in \cite{metdj1} (Theorem \ref{mjk} in this paper). 
We will prove that set $S$ is also a doubly resolving set.

The construction is the following.

Let $S'= \{ \{3i-2,3i-1\},\{3i-1,3i\} | 1 \le i \le t\}$, where $t = \lfloor \frac{n}{3} \rfloor$.\\

Now 

\begin{equation} \label{dbas}
S= \begin{cases} S' , \quad n=3t \\ 
S' \bigcup \{  \{3t-1,3t+1\} \},  \quad n=3t+1 \\
S' \bigcup \{  \{3t-1,3t+1\},   \{3t-1,3t+2\} \}, \quad n=3t+2 \\
\end{cases}
\end{equation} . \\

Let $P=(P_1,...,P_t)$ be a partition of indices \{1,2,...,n\}, 
such that \\ $P_i = \begin{cases} \{3i-2,3i-1,3i\}, \quad  i<t \\ P \setminus \bigcup\limits_{j = 1}^{t - 1} {{P_j}} ,  \quad i=t   \end{cases}$.\\

In Theorem 3.32 in \cite{metdj1} it is proved that $S$ is a resolving set, i.e. $(\forall A,B \in V(J_{n,2})) (A \ne B) (\exists X \in S)$
 $d(A,X) \ne d(B,X)$.

Let us prove that $(\forall A,B \in V(J_{n,2})) (A \ne B) (\exists Y \in S) \quad d(A,Y) = d(B,Y)$.
Let $A = \{a,b\}$ and $B=\{c,d\}$, such that $a \in P_{i_a}$, $b \in P_{i_b}$, $c \in P_{i_c}$ and
$d \in P_{i_d}$. It should be noted that some indices $i_a,i_b,i_c,i_d$ can be equal.
Therefore, set $I = \{i_a,i_b,i_c,i_d\}$ has cardinality $|I| \le 4$.
Since $n \ge 15$ then $t \ge 5$, and there exists $i_y$ such that $1 \le i_y \le t$ and $i_y \notin \{i_a,i_b,i_c,i_d\}$.
For $Y = \{3 i_y -2, 3 i_y-1\}$ it holds that $A \bigcap Y = B \bigcap Y = \emptyset$ implying $d(A,Y)=d(B,Y)=2$. \\

Since  $d(A,X) \ne d(B,X)$ and $d(A,Y)=d(B,Y)$ it follows $d(A,X) - d(A,Y) \ne d(B,X) - d(B,Y)$,
which implies that set $S$ is a dobly resolving set, and $\psi(J_{n,2}) \le \lceil \frac{2n}{3} \rceil$.
\end{proof} 

\begin{obs} \label{j2o} For $4 \le n \le 14$ it holds  $\psi(J_{n,2}) = \lceil \frac{2n}{3} \rceil$.
This result is obtained by Gurobi 13.0 using the ILP formulation (\ref{ilpdk})-(\ref{ilpdk2}).
The Gurobi running time grows rapidly with the dimension,
and for $n=14$ it is 5530 seconds. It should be noted that the doubly metric basis (\ref{dbas})
from Theorem \ref{j2} is also a doubly metric basis of $J_{n,2}$ for $4 \le n \le 14$.
\end{obs}

The next Corollary combines Theorem \ref{j2} and Observation \ref{j2o}:  

\begin{cor} \label{j2b} For $n \ge 4$ it holds  $\psi(J_{n,2}) = \lceil \frac{2n}{3} \rceil$.
\end{cor}

The next theorem obtains the exact value of the doubly metric dimension of $K_{n,2}$.

\begin{thm} \label{k2} For $n \ge 15$ it holds  $\psi(K_{n,2}) = \lceil \frac{2n}{3} \rceil$.
\end{thm}
\begin{proof} By Theorem \ref{mjk} from \cite{metdj1} it holds $\beta(K_{n,2}) = \lceil \frac{2n}{3} \rceil$.
Having in mind that for every graph $G$, $\psi(G) \ge \beta(G)$, so $\psi(K_{n,2}) \ge \lceil \frac{2n}{3} \rceil$.

Let $S$ be the set defined by (\ref{dbas}) from Theorem \ref{j2}.
As already noticed $\overline{J_{n,2}} = K_{n,2}$ and $Diam(J_{n,2}) = Diam(K_{n,2}) = 2$, so $V(J_{n,2}) = V(K_{n,2})$
and for each two vertices $A,B \in V(K_{n,2})$ with $A \ne B$ it holds $d(A,B) = 3 - d_{J_{n,2}}(A,B)$,
where $d(A,B)$ and  $d_{J_{n,2}}(A,B)$ are distances between $A$ and $B$ in Kneser graph $K_{n,2}$ and
Johnson graph $J_{n,2}$, respectively. 

In Theorem \ref{j2} it is proved $$(\forall A,B \in V(J_{n,2}) = V(K_{n,2})) (\exists X \in S) \quad A \bigcap X = B \bigcap X = \emptyset$$
This implies $d(A,X)=d(B,X)=1$.\\
In Theorem \ref{j2} it is proved $$(\forall A,B \in V(J_{n,2}) = V(K_{n,2}))(A \ne B) (\exists Y \in S) \quad d_{J_{n,2}}(A,Y) \ne d_{J_{n,2}}(B,Y)$$
If $A \ne Y$ and $B \ne Y$, from $d(A,Y) = 3 - d_{J_{n,2}}(A,Y)$ and $d(B,Y) = 3 - d_{J_{n,2}}(B,Y)$ it follows $d(A,Y) \ne d(B,Y)$.

In the case when $A = Y$ then $0 = d(A,Y) \ne d(B,Y) \ge 1$. Similarly, if $B = Y$ then $0 = d(B,Y) \ne d(A,Y) \ge 1$

Therefore, $d(A,X)-d(A,Y) \ne d(B,X)-d(B,Y)$ implying that set $S$ is a doubly resolving set for graph $K_{n,2}$, which implies $\psi(K_{n,2}) \le \lceil \frac{2n}{3} \rceil$.  
\end{proof}

\begin{obs} \label{k2a} For $6 \le n \le 14$ it holds  $\psi(K_{n,2}) = \lceil \frac{2n}{3} \rceil$.
This result is obtained by Gurobi 13.0 using the ILP formulation (\ref{ilpdk})-(\ref{ilpdk2}).
The Gurobi running time grows rapidly with the dimension,
and for $n=14$ it is 7223 seconds. It should be noted that the doubly metric basis (\ref{dbas})
from Theorem \ref{j2} is also a doubly metric basis of $K_{n,2}$ for $6 \le n \le 14$.
Finally, $\psi(K_{5,2}) = 3$ with a doubly metric basis $\{ \{1,2\}, \{1,3\}, \{1,4\} \}$
is obtained also by Gurobi 13.0 for 0.06 seconds.
\end{obs}

The next Corollary combines Theorem \ref{k2} and Observation \ref{k2a}:  

\begin{cor} \label{k2b} For $n \ge 6$ it holds  $\psi(K_{n,2}) = \lceil \frac{2n}{3} \rceil$.
\end{cor}

\section{Conclusions}

The doubly and strong metric dimensions of both Johnson and Kneser graphs have been studied. 
Exact values $\beta_S(J_{n,k}) = \binom{n-1}{k}$ and $\psi(J_{n,2}) = \psi(K_{n,2}) = \lceil \frac{2n}{3} \rceil$
are obtained. Furthermore, it is shown that for $n \ge 3k-1$, the value of $\beta_S(K_{n,k})$ is equal to $\binom{n}{k} - \lfloor \frac{n}{k} \rfloor$.

In future research, it would be interesting to investigate the values of $\beta_S(K_{n,k})$ for $n < 3k-1$,
as well as those of $\psi(J_{n,k})$ and $\psi(K_{n,k})$ for $k \ge 3$.
Furthermore, efforts could be directed towards determining the strong and/or doubly metric dimensions of other interesting classes of graphs. 
 
\bibliographystyle{elsarticle-num}
 \bibliography{paper}

\end{document}